\documentclass[11pt,a4paper]{siamart1116}

\usepackage{amsmath,amssymb,amsbsy,amsfonts,latexsym,amsopn,amstext,%
                                               amsxtra,euscript,amscd}
\usepackage{mathrsfs}
\usepackage{algorithm,algpseudocode}
\usepackage{url,ifthen}
\usepackage{multirow,rotating,tabularx}
\usepackage{microtype}
\usepackage{xspace}
\usepackage{tikz,tikz-cd}
\usepackage{enumitem}
\usepackage[hmargin=3.5cm]{geometry}

\newcommand{\Pj}{\mathbb{P}}

\newcommand{\F}{\mathbb{F}}

\newcommand{\Z}{\mathbb{Z}}
\newcommand{\Fl}{\mathbb{F}\/_\ell}
\newcommand{\Fq}{\mathbb{F}\/_q}
\newcommand{\Fqd}{\mathbb{F}\/_{q^d}}
\newcommand{\Fp}{\mathbb{F}\/_p}

\DeclareMathOperator{\Gal}{Gal}

\DeclareMathOperator{\monic}{monic}
\DeclareMathOperator{\res}{res}

\tikzset{commutative diagrams/diagrams={row sep=large,column sep=large}}

\title{Constructing Permutation Rational Functions From Isogenies}

\author{%
Gaetan Bisson%
\thanks{University of French Polynesia (\email{bisson@gaati.org})}%
\and
Mehdi Tibouchi%
\thanks{NTT Secure Platform Laboratories (\email{tibouchi.mehdi@lab.ntt.co.jp})}%
}

\begin{document}

\maketitle

\begin{abstract}
A \emph{permutation rational function} $f\in \Fq(x)$ is a rational
function that induces a bijection on $\Fq$, that is, for all $y\in\Fq$
there exists exactly one $x\in\Fq$ such that $f(x)=y$. Permutation
rational functions are intimately related to exceptional rational
functions, and more generally exceptional covers of the projective line,
of which they form the first important example.

\smallskip
In this paper, we show how to efficiently generate many permutation
rational functions over large finite fields using isogenies of elliptic
curves, and discuss some cryptographic applications. Our algorithm is based on Fried's modular interpretation of
certain dihedral exceptional covers of the projective line
(\emph{Cont. Math.}, 1994).
\end{abstract}

\begin{keywords}
Permutation rational functions, Exceptional covers, Isogenies, Elliptic
curves, Cryptography.
\end{keywords}

\begin{AMS}
11T71, 14K02
\end{AMS}

\section{Introduction}
\label{s:intro}

A map $X\to Y$ of (smooth, projective) algebraic curves over a finite
field $\Fq$ is called an \emph{exceptional cover} when the induced map on
$\F_{q^t}$-points $X(\F_{q^t})\to Y(\F_{q^t})$ is a bijection for
infinitely many values of $t$ (necessarily including $t=1$). The
construction of exceptional covers is an important problem in arithmetic
algebraic geometry~\cite{Fri05}, which also has applications to
cryptography.

One can in particular mention the construction of hash functions with
values in algebraic curves and their Jacobians: while there already is
abundant literature on the subject, the construction techniques
proposed so far have been somewhat ad hoc and
unsystematic~\cite{Tib2011-PhD}, and only
partial results have been obtained for curves of genus
$\geq 2$~\cite{DBLP:conf/pairing/KammererLR10,DBLP:conf/pairing/FouqueT10,%
1310.1013}. A more programmatic approach has been suggested
in~\cite{Tib2013-ECC} based on the observation that, given an
exceptional cover $X\to\Pj^1$ of the projective line over $\Fq$,
one can obtain encodings of elements of $\Fq$ to all curves $Y$ with a
non-constant map $h\colon X\to Y$ simply by composing the bijection
$\Pj^1(\Fq)\to X(\Fq)$ with $h$. The construction of hash functions is
thus reduced to obtaining explicit exceptional covers of the projective
line.

The simplest such exceptional covers are those of genus zero, namely
rational functions $f\in\Fq(x)$ inducing a permutation of
$\Pj^1(\F_{q^t})$ for infinitely many $t$. They are not directly
applicable to hashing (since any curve $Y$ with a non-constant map
$\Pj^1\to Y$ is rational), but they are comparably
well-understood~\cite{Fri94,GMS03} and an interesting first step towards
the general case. Fried~\cite[\S4]{Fri05} also suggested that these
exceptional covers should play an important role in public-key
cryptography.

Indeed, exceptional covers are essentially the same objects
as \emph{permutation rational functions}, {i.e.} rational functions over
$\Fq$ inducing a bijection of $\Fq$ to itself: clearly, an exceptional
cover $f\colon\Pj^1\to\Pj^1$ satisfying $f(\infty)=\infty$ (which can
always be satisfied up to a linear fractional transformation) is a
permutation rational function, and one can show that if the degree of $f$
is small compared to $q$, the converse also
holds~\cite{GTZ07,DBLP:conf/provsec/Tibouchi14}.
In particular, we can see these rational functions as generalizations of
the RSA polynomials $x^e$ with $e$ coprime to $q-1$.

\paragraph{Our contributions.} Based on Fried's modular interpretation of
a large class of exceptional covers $\Pj^1\to\Pj^1$ called exceptional
involution covers~\cite[Cor. 3.5]{Fri94}, we describe an algorithm to
generate permutation rational functions of any constant prime degree
$\ell\geq 5$ (which are, in fact, exceptional covers) over large finite
fields and show that it is efficient and
practical. We also expand upon the RSA analogy alluded to in
\cite{Fri05} and discuss how our algorithm might indeed be used
to obtain new factoring-related trapdoor permutations that behave better
than the RSA trapdoor permutation against certain classes of attacks.

\section{Permutation rational functions from isogenies}
\label{sec:isogenies}

Consider two elliptic curves $E\colon y^2=x^3+ax+b$ and $E'\colon
y^2=x^3+a'x+b'$ over a finite field $\Fq$ of characteristic $\neq2,3$ and an
isogeny $\varphi\colon E\to E'$ defined over $\Fq$. Since $\varphi$
commutes to the involution of multiplication by $-1$, the $x$-coordinate
(resp.{} $y$-coordinate)
of $\varphi(x,y)$ is an even (resp.{} odd) function of $y$. And since
$y^2=x^3+ax+b$, this means that there exist unique
rational functions $u_\varphi,v_\varphi\in\Fq(x)$ such
that $\varphi$ has the form:
\[ \varphi(x,y) = \big( u_\varphi(x), y\cdot v_\varphi(x) \big). \]
This paper is based on the following observation.

\begin{theorem}
\label{th:main}
Let $\varphi\colon E\to E'$ as above be an isogeny defined over
$\Fq$ of prime degree $\ell$. The following conditions are equivalent.
\begin{enumerate}[label={\upshape({\itshape \roman*}\/)},leftmargin=*]
\item $u_\varphi$ has no $\Fq$-rational pole;
\item the kernel $K$ of $\varphi$ satisfies $K(\F_{q^2}) = \{0\}$;
\item $u_\varphi$ is a permutation rational function.
\end{enumerate}
\end{theorem}
\begin{proof}
$\text{\itshape(i)}\Rightarrow\text{\itshape(ii)}$.
Let $P=(x,y)$ be a non-identity element of
$K(\F_{q^2})$. We know that $x$ is a pole of $u_\varphi$ (as $P\in K$)
and will now show that $x\in\Fq$.
First note that $P$ is a generator of $K$ since that group has prime order $\ell$.
As a result, in view of the fact that
$\varphi$ commutes with the Frobenius, $P$ must be an eigenvector
of the Frobenius $F$ of $E$ for some eigenvalue $\lambda$ (when we view
$F$ as a linear endomorphism of the $\F_\ell$-vector space
$E[\ell](\overline{\Fq})$). Moreover,
since $P$ is in $E(\F_{q^2})$, we have $F^2(P) = P$, hence
$\lambda=\pm1$. Therefore, $F(P)=(x,\pm y)$ and in particular $x^q=x$,
that is, $x\in\Fq$. Thus, $x$ is an $\Fq$-rational pole of $u_\varphi$.

$\text{\itshape(ii)}\Rightarrow\text{\itshape(iii)}$.
The rational fraction $u_\varphi$ is a permutation rational function
if and only if it has no rational pole and is injective.
Assuming $K(\F_{q^2}) = \{0\}$, $u_\varphi$ cannot have
a rational pole: indeed, if $x$ were such a pole, a point $P$ on $E$ with that
$x$-coordinate would be defined over $\F_{q^2}$ and satisfy
$\varphi(P)=0$. Suppose now that there exist $x,x'\in\Fq$ such that $u_\varphi(x)=u_\varphi(x')$.
Take two points $P,P'\in E(\F_{q^2})$ having these values
as their respective $x$-coordinates. Since
$u_\varphi(x)=u_\varphi(x')$, we have $\varphi(P)=\pm\varphi(P')$, and so
$P\mp P'\in K(\F_{q^2}) = \{0\}$. This implies that $P=\pm P'$ and thus $x=x'$.
Therefore $u_\varphi$ is injective.

$\text{\itshape(iii)}\Rightarrow\text{\itshape(i)}$ is clear.
\qed
\end{proof}

We note that the condition in the theorem can only be satisfied for $\ell\geq 5$.
Indeed, since the kernel of $\varphi$ consists of $\ell$ points including
the point at infinity, the denominator of $u_\varphi$ is of degree
$\ell-1$. In particular, for $\ell=2$, it is linear and thus does have a
rational root. Moreover, for $\ell$ odd, the non-zero kernel points come
in pairs $\{\pm P\}$ of distinct points with the same $x$-coordinate, so
the denominator of $u_\varphi$ is actually the square of a
polynomial of degree $(\ell-1)/2$. This again implies that
$u_\varphi$ has a rational pole for $\ell=3$. On the other hand, we will
be able to construct examples of the situation in the theorem for any
$\ell\geq 5$.

Note also that under the conditions of the theorem, $u_\varphi$ is in
fact an exceptional cover $\Pj^1\to\Pj^1$. This follows from the fact
that $K(\F_{q^{2t}})$ remains trivial for any $t$ coprime to the
degree of the finite extension of $\Fq$ over which the points of $K$ are
defined.

The above theorem enables us to efficiently construct permutation rational
functions of given prime degree $\ell\geq 5$ over prescribed finite fields
$\Fq$ of cryptographic size. To do so, we proceed as follows.

\section{Computing isogeny kernels}

As before, let $E$ be an elliptic curve defined over a finite field $\Fq$.
Denote by $P$ a point of prime order $\ell$ and by $K$ the subgroup it
generates. The isogeny $E\to E/K$ satisfies the conditions of
Theorem~\ref{th:main} if and only if $K$ is rational and $K(\F_{q^2})=\{0\}$.
The second condition is easy to test: since $\ell$ is prime, all nontrivial
points of $K$ generate all others; it thus suffices to verify that $P$ is not
defined over $\F_{q^2}$. To efficiently test whether $K$ is rational we use the
following criterion.

\begin{lemma}
Let $P$ be a point of prime order $\ell$ on $E$. Denote by $d$ the degree of
the field extension $\Fq(x_P)/\Fq$. Let $\tau$ be an integer of order exactly
$d$ in $\Fl^\times/\{\pm 1\}$. The subgroup $K$ generated by $P$ is rational if
and only $x_{[\tau]P}$ is a Galois conjugate of $x_P$.
\end{lemma}

\begin{proof}
The subgroup $K$ is stable under the involution of multiplication by $-1$ and
is thus completely determined by the set $H$ of $x$-coordinates of its
nontrivial points, which forms a principal homogeneous space for
$\Fl^\times/\{\pm 1\}$ where $\lambda\in\Fl^\times$ acts by $x_P\mapsto
x_{[\lambda]P}$.

The Frobenius automorphism $\pi$ of $\Fqd/\Fq$ stabilizes $H$ and thus acts as
an element $\lambda_\pi$. Now let $e$ denotes the degree of the smallest
extension of $\F_q$ over which $H$ is defined. Note that
$\F_{q^e}$ is also the field of definition of $K$ since this subgroup has odd order.
Then $\lambda_\pi$ is of order
exactly $d/e$. Therefore, the group $\Gal(\Fqd/\Fq)$ embeds in
$\Fl^\times/\{\pm 1\}$ as $\mu_{d/e}$ and its action partitions $H$ into orbits
of length $d/e$. The stabilizer of $P$'s orbit is then $\mu_{d/e}$. In
particular, $e=1$ if and only if $x_{[\tau]P}$ lies in the same orbit as $x_P$.
\qed
\end{proof}

To make the above criterion explicit, recall that multiplication-by-$k$ is an
algebraic map on $E$
\[
P=(x,y)\longmapsto
[k]P=\left(\frac{\phi_k(x)}{\psi_k(x)^2},\frac{\omega_k(x,y)}{\psi_k(x)^3}\right)
\]
where the polynomials $\phi_k$, $\psi_k$, and $\omega_k$ are efficiently
computable. It follows that $x$-coordinates of $\ell$-torsion points are roots
of the so-called $\ell$-division polynomial $\psi_\ell(x)$.
If $f(x)$ is a degree-$d$ irreducible factor of $\psi_\ell(x)$, we can test
whether its roots are the $x$-coordinates of points $P$ such that $f(x_{[\tau]P})=0$
by checking whether
\[
f\left(\frac{\phi_\tau(x)}{\psi_\tau(x)^2}\right)=0\bmod f(x).
\]
In that case, the $x$-coordinates of other points of $K$ are obtained through
the map $x_P\mapsto x_{[\rho]P}$ for $\rho$ in $\Fl^\times/\{\pm 1\}/\mu_d$. We
can compute them as the roots of
\[
\gcd\left(\psi_\ell(x), g_\rho(x)\right)
\qquad
\text{where}
\qquad
g_\rho(x)=\res_y\left(f(y), \phi_\rho(x)-\psi_\rho(x)^2 y\right).
\]
Indeed, for each $P\in K$ there
are $\rho^2$ points $Q$ such that $[\rho]Q=P$, only one of which lies in $K$;
the others have order $\ell\rho$ and are eliminated by taking the $\gcd$ with
$\psi_\ell(x)$.

Note that the above is unnecessary for $d=\frac{\ell-1}{2}$. In that case, the
condition of the lemma always holds and degree-$d$ irreducible factors of
$\psi_\ell(x)$ are directly kernel polynomials of rational degree-$\ell$
isogenies.

Putting the above together we obtain Algorithm~\ref{alg:kernels}.

\renewcommand{\algorithmicrequire}{\textbf{Input:}}
\renewcommand{\algorithmicensure}{\textbf{Output:}}
\begin{algorithm}[t]
\begin{algorithmic}[1]
  \Require an elliptic curve $E/\Fq$ with $q=p^\alpha$, $p\neq 2,3$, and a prime $\ell\notin\{2,3,p\}$.
  \Ensure the list $\mathscr{L}$ of all kernel polynomials of $\ell$-isogenies satisfying the conditions of Theorem~\ref{th:main}.

  \State compute the $\ell$-division polynomial $\psi_\ell(x)$ of $E/\Fq$.
  \State let $\omega$ denote a generator of $\Fl^\times$.
  \State initialize $\mathscr{L}$ to the empty list.
  \For{each positive divisor $d$ of $\frac{\ell-1}{2}$ other than $1$}
    \State compute the set $\mathscr{F}$ of degree-$d$ irreducible factors of $\psi_\ell(x)$.
    \State let $\tau$ be the smallest positive integer such that
    $\tau\equiv\pm\omega^{\frac{\ell-1}{2d}}\bmod \ell$.
    \State using a remainder tree, compute $a_f =
           \phi_\tau(x)/\psi_\tau(x)^2 \bmod f(x)$ for each
           $f\in\mathscr{F}$.
    \While{$\mathscr{F}$ is not empty}
      \State remove the first element from $\mathscr{F}$ and call it $f$.
      \If{$f(a_f)\neq0\bmod f(x)$}
        \State \textbf{continue} to the next iteration.
      \EndIf
      \State let $k\leftarrow f$.
      \For{$m=1$ to $\frac{\ell-1}{2d}-1$}
        \State let $\rho$ be the smallest positive integer such that
	       $\rho\equiv\pm\omega^m\bmod\ell$.
        \State let $g(x)\leftarrow\res_y\left(f(y), \phi_\rho(x)-\psi_\rho(x)^2 y\right)$.
        \State let $g(x)\leftarrow \monic\left(\gcd(\psi_\ell(x),g(x))\right)$.
	\State let $k\gets g\cdot k$.
	\State remove $g(x)$ from $\mathscr{F}$.
      \EndFor
      \State add $k(x)$ to the list $\mathscr{L}$.
    \EndWhile
  \EndFor
  \State \textbf{return} $\mathscr{L}$.
\end{algorithmic}
\caption{Compute kernel polynomials satisfying Theorem~\ref{th:main}.}
\label{alg:kernels}
\end{algorithm}

\begin{theorem}
Algorithm~\ref{alg:kernels} has a probabilistic running time of $\widetilde{O}(\ell^{4+\epsilon}\log(q)^2)$.
\end{theorem}

\begin{proof}
Step~1 computes the division polynomial $\psi_\ell(x)\in\Fq[x]$ which is of
degree $\frac{\ell^2-1}{2}$. Using the formulas of \cite[page 200]{weber-cm}
and an asymptotically fast method for polynomial multiplication, this takes
quasi linear time in the output: $O(\ell^2\log(q))$.

The loop from Step~4 runs a maximum of $O(\ell^\epsilon)$ iterations, for any
$\epsilon>0$, thanks to the bound on the number of divisors from \cite[Theorem
13.12]{apostol}.

In order to find the irreducible factors of degree $d<\frac{\ell-1}{2}$ of
$\psi_\ell(x)$ we use ideas from the Cantor--Zassenhaus algorithm: first we evaluate
\[
r(x)=\frac{\gcd\big(\psi_\ell(x),x^{q^d}-x\big)}{\gcd\big(\psi_\ell(x),x^{q}-x\big)}
\]
which is the product of all such factors; then we isolate those factors by
iteratively splitting $r(x)$ as
\[
r(x)=g(x)\cdot\frac{r(x)}{g(x)}
\quad
\text{where}
\quad
g(x)=\gcd\left(r(x),h(x)^{\frac{q^d-1}{2}}\right)
\]
with $h(x)$ drawn uniformly at random from $\Fq[x]/(r(x))$. Evaluating such
expressions boils down to computing $O(q^\ell)$-powers in $\Fq[x]/(r(x))$
which, since $r(x)$ has degree $O(\ell^2)$, gives an asymptotic
complexity of $\widetilde{O}(\ell^3\log(q)^2)$ for Step~5.

Step~7 takes at most $\ell^{2+\epsilon}$ elementary operations in $\F_q$
using \cite{mobo72}.

Both Step~16 and 17 use $O(d\ell^2)$ operations over $\F_q$. They run at most
once per element $f\in\mathscr{F}$, which gives an overall contribution of
$\widetilde{O}(\ell^4\log(q))$ to the running time. This dominates the loop
from Step~8 to 22.
\qed
\end{proof}

\section{Computing permutation rational functions}

We now turn to our main algorithm.
First recall that isomorphism classes of elliptic curves can be uniquely
identified by their $j$-invariant. Under this map, pairs of $\ell$-isogenous
elliptic curves $(E,E')$ are completely characterized by the equality
$\Phi_\ell(j(E),j(E'))=0$ where $\Phi_\ell(X,Y)$ denotes the $\ell$-modular
polynomial. Thus, to select a rational $\ell$-isogeny $E\to E'$, we simply draw
$j(E)$ uniformly at random from $\Fq$ until $\Phi_\ell(X,j(E))$ has a root.

For all suitable kernel polynomials $f(x)$ found by Algorithm~\ref{alg:kernels},
we output the corresponding
isogeny map derived using Kohel's formula \cite[Section 2.4]{kohel-phd}.
This gives Algorithm~\ref{alg:friedpoly}.

\renewcommand{\algorithmicrequire}{\textbf{Input:}}
\renewcommand{\algorithmicensure}{\textbf{Output:}}
\begin{algorithm}[t]
\begin{algorithmic}[1]
  \Require a prime power $q=p^\alpha$ with $p\neq 2,3$ and a prime $\ell\notin\{2,3,p\}$.
  \Ensure permutation rational functions of degree $\ell$ over $\Fq$.

  \State compute the reduction to $\Fq[X]$ of the $\ell$-modular polynomial $\Phi_\ell(X,Y)$.
  \Loop
    \Repeat
      \State draw an element $j\in\Fq$ uniformly at random.
    \Until{the polynomial $\Phi_\ell(X,j)$ has at least one root.}
    \State let $E/\Fq$ denote an elliptic curve with $j$-invariant $j$.
    \For{each polynomial $f(x)$ output by Algorithm~\ref{alg:kernels}}
      \State compute the isogeny $\varphi$ with kernel polynomial $f(x)$ using Kohel's formula.
      \State \textbf{return} its $x$-coordinate map $u_\varphi$.
    \EndFor
  \EndLoop
\end{algorithmic}
\caption{Compute permutation rational functions.}
\label{alg:friedpoly}
\end{algorithm}

\begin{theorem}
Heuristically, Algorithm~\ref{alg:friedpoly} runs in $\widetilde{O}(\ell^{4+\epsilon}\log(q)^2)$ time.
\end{theorem}

\begin{proof}
Step~1 uses the method of \cite{drew-modpol} to compute the modular
polynomial in $\widetilde{O}(\ell^3\log(q))$ operations. For Step~7
we refer to Algorithm~\ref{alg:kernels}.
Finally, the complexity of Step~8 is quasi linear in its input.

To conclude the proof, we only need to show that the average number of loop iterations
is bounded. First consider the innermost loop. The probability that a random
$j$-invariant satisfies the condition in Step~5 is exactly that of the
corresponding elliptic curve $E$ having a rational degree-$\ell$ isogeny.
Assuming the curve is ordinary, which only disregards finitely many
$j$-invariants, the so-called volcano structure 
\cite{kohel-phd,fouquet-morain} implies that, if the discriminant $\Delta(E)$ is a
nonzero square modulo $\ell$, then $E$ has $\ell$-maximal endomorphism ring and
is connected to other such curves by a cycle of $\ell$-isogenies.
Therefore the probability that we so obtain a rational isogeny is
$\frac{\ell-1}{2\ell}$ under the heuristic assumption that $\Delta(E)$ behaves modulo $\ell$ as a
random integer.

The outermost loop is executed as many times as we require rational isogenies
before the kernel of one admits no nontrivial rational point. Recall that the
modular curves $X_0(\ell)$ and $X_1(\ell)$ essentially parametrize
pairs $(E,K)$ where $K$ is an order-$\ell$ subgroup of the elliptic curve $E$,
and pairs $(E,P)$ where $P$ is an order-$\ell$ point of $E$,
respectively. Similarly, one can consider the modular curve associated with the congruence
subgroup
\[ \Gamma_1'(\ell) = \left\{ \begin{bmatrix} a & b \\ c & d \end{bmatrix} \equiv
   \begin{bmatrix} \pm1 & \ast \\ 0 & \ast \end{bmatrix}\pmod\ell \right\} \]
whose associated modular curve parametrizes pairs $(E, \pm P)$ consisting of an
elliptic curve and a point of order $\ell$ up to sign. The natural map $X_1(\ell) \to 
X\big(\Gamma'_1(\ell)\big)$ is actually an isomorphism, and 
since $X_1(\ell)\to X_0(\ell)$ is a cyclic Galois cover of degree
$\frac{\ell-1}{2}$, Chebotarev's density theorem shows that the image of $X\big(\Gamma'_1(\ell)\big)(\Fq)\to
X_0(\ell)(\Fq)$ has density $\frac{2}{\ell-1}+O(q^{-1/2})$.
Thus, if $(E,K)$ is uniformly distributed, the probability
that all points $P$ of $K$ are defined over $\F_{q^2}$ (which is equivalent to saying that $\{\pm P\}$ is defined
over $\Fq$) converges to $\frac{2}{\ell-1}$. In
particular, if $E$ admits a rational subgroup $K$ of order $\ell$,
the probability that one such subgroup does not have nontrivial $\F_{q^2}$-points
is asymptotically $1-\frac{2}{\ell-1}$.

We conclude that the overall success probability of an iteration is at
least $\frac{\ell-1}{2\ell}\big(1-\frac{2}{\ell-1}\big) =
\frac{\ell-3}{2\ell}$ up to the $O(q^{-1/2})$ error term,
hence the expected number of
iterations is less than about $\frac{2\ell}{\ell-3} \leq 5$.
\qed
\end{proof}

As an example, take $q=127$ and $\ell=13$.
For $j=60$ we find that $\Phi_\ell(X,j)$ has two roots in $\Fq$.
Thus, any elliptic curve $E$ with $j(E)=60$ is the domain of two
rational degree-$\ell$ isogenies; we take $E:y^2=x^3+25x+58$,
of which the $\ell$-division polynomial factors as
$\psi_\ell(x)=f_1(x)f_2(x)f_3(x)Q(x)$ with
\begin{align*}
f_1(x) &= x^3 + 88 x^2 + 60 x + 59, \\
f_2(x) &= x^3 + 91 x^2 + 14 x + 57, \\
f_3(x) &= x^6 + 36 x^5 + 17 x^4 + 73 x^3 + 88 x^2 + 11 x + 31,
\end{align*}
and $Q(x)$ is the product of six irreducible polynomials of degree twelve.

Since $f_3(x)$ has degree $\frac{\ell-1}{2}$, it is the kernel polynomial of a
degree-$\ell$ isogeny with no rational kernel point. The $x$-coordinate map of
that isogeny provides a first permutation rational fraction:
\begin{align*}
x\longmapsto\big(
	x^{13} + 72 x^{12} + 84 x^{11} + 72 x^{10} + 2 x^9 + 15 x^8 + 91 x^7 & \\
	+ 94 x^6 + 4 x^5 + 66 x^4 + 17 x^3 + 49 x^2 + 48 x + 53 &
\big)\big/f_3(x)^2.
\end{align*}

Now consider $f_1(x)$. We take $\omega=2$ as a generator of $\Fl^\times$ and
deduce $\tau=4$. The condition
$f_1\left(\frac{\phi_\tau(x)}{\psi_\tau(x)^2}\right)=0\bmod f_1(x)$ holds and
therefore $f_1(x)$ is a factor of the kernel polynomial of a rational
degree-$\ell$ isogeny. We compute the other factor as Step~15--17 in
Algorithm~\ref{alg:kernels} with $m=1$ and find $g(x)=f_2(x)$. We then compute
the isogeny with kernel polynomial $f_1(x)f_2(x)$ and obtain a second permutation rational fraction
as its $x$-coordinate map:
\begin{align*}
x\longmapsto\big(
	x^{13} + 67 x^{12} + 13 x^{11} + 61 x^{10} + 83 x^9 + 50 x^8 + 49 x^7 & \\
	+ 80 x^6 + 75 x^5 + 88 x^4 + 7 x^3 + 41 x^2 + 38 x + 7 &
\big)\big/\big(f_1(x)f_2(x)\big)^2.
\end{align*}

Table~\ref{table:timings} reports on running times for a simple PARI/GP
\cite{pari} implementation of Algorithm~\ref{alg:friedpoly} on a single core of
an Intel Xeon E3-1275 CPU.

Additionally, Table~\ref{table:density} gives the average density of computed
kernel polynomials which are not irreducible, that is, for which
$d<\frac{\ell-1}{2}$. It shows that, for certain values of $\ell$, although the
special case of $d=\frac{\ell-1}{2}$ (that is, irreducible kernel polynomials) greatly simplifies our
algorithms, it significantly restricts the range of permutation rational
fractions found. Note that a density of zero is expected for $\ell=23$ and
$\ell=59$ since in that case $\frac{\ell-1}{2}$ is exactly twice a prime number
$p$, so all isogenies of degree $d=2,p$ have rational kernel points.

\begin{table}
\begin{center}
\begin{tabular}{|l||r|r|r|r|}
\hline
$q=$ & $2^{127}-1$ & $2^{255}-19$ & $2^{511}-187$ & $2^{1023}-361$
\\
\hline
$\ell=13$ & 0.13 & 0.24 & 0.57 & 2.53 \\
$\ell=23$ & 0.68 & 1.28 & 2.98 & 9.81 \\
$\ell=37$ & 3.25 & 5.99 & 15.48 & 43.14 \\
$\ell=59$ & 21.38 & 35.02 & 89.05 & 227.20 \\
\hline
\end{tabular}
\end{center}
\caption{Average running time in seconds for Algorithm~\ref{alg:friedpoly}.}
\label{table:timings}
\end{table}

\begin{table}
\begin{center}
\begin{tabular}{|l||r|r|r|r|}
\hline
$q=$ & $2^{127}-1$ & $2^{255}-19$ & $2^{511}-187$ & $2^{1023}-361$
\\
\hline
$\ell=13$ & 0.55 & 0.40 & 0.55 & 0.50 \\
$\ell=23$ & --- & --- & --- & --- \\
$\ell=37$ & 0.70 & 0.80 & 0.55 & 0.60 \\
$\ell=59$ & --- & --- & --- & --- \\
\hline
\end{tabular}
\end{center}
\caption{Density of computed kernel polynomials with $d<\frac{\ell-1}{2}$.}
\label{table:density}
\end{table}

\section{A family of candidate trapdoor permutations}
\label{sec:trapdoor}

Using the algorithm of the previous section, one can obtain a permutation
rational function analogue of the RSA trapdoor permutation. Indeed,
consider an RSA modulus $N=p\cdot q$. With the knowledge of the
factorization of $N$, one can efficiently generate permutation rational
functions $u = a/b \in \Fp(x)$ and $v=c/d \in \Fq(x)$ of the same prime
degree $\ell$, and use the Chinese Remainder Theorem to deduce
polynomials $r,s\in \Z[x]$ of degree at most $\ell$ with coefficients in
$(-N/2,N/2)$ such that $u=r/s\bmod p$ and $v=r/s\bmod q$.

The function $x\mapsto r(x)/s(x) \bmod N$ is then a permutation of
$\Z/N\Z$ which is easy to invert with the knowledge of the factorization
of $N$ (simply reduce modulo $p$ and $q$ and use an algorithm like
Berlekamp or Cantor--Zassenhaus to invert $u$ and $v$). However, it seems
hard to invert it otherwise.

This construction is somewhat less efficient in terms of public key size
and evaluation efficiency than the RSA trapdoor permutation, but it seems
to resist certain types of attacks better: for example, there are no
obvious malleability properties, which should thwart most types of
blinding attacks or related message attacks~\cite{Bon99}.

On the other hand, the security analysis is not entirely straightforward.
Publishing $r$ and $s$ could reveal some information on the factorization
of $N$, since its factors belong to the (presumably sparse!) set of
primes $p_0$ such that $\lambda r + \mu s$ has exactly one root modulo $p_0$
for all integers $\lambda,\mu$, $\lambda$ coprime to $N$. For example, if
many values of $(\lambda,\mu)$ provided congruence conditions on $p_0$,
one might be able to recover $p$ and $q$ using the Chinese Remainder
Theorem. In practice, however, the polynomial $\lambda r + \mu s\in\Z[x]$ will
typically have Galois group $S_\ell$, and so one presumably cannot hope to obtain a
really effective description of the set of primes at which it has a root.

\section{Conclusion and open problems}

We have seen that generating permutation rational functions, or
exceptional covers of genus zero of the projective line, could be done
quite practically using elliptic curve isogenies, even over finite fields
of cryptographic size. The covers we obtain with our algorithms are (up
to conjugation by linear fractional transformations) exactly the
exceptional involution covers defined by Fried in~\cite[\S3.2]{Fri94}.
Since the classification of genus-zero exceptional covers of the
projective line has been given by Guralnick {et al.}~\cite{GMS03}, one
could ask how to effectively generate permutation rational functions from
the remaining families.

Perhaps more importantly, one important open question related to this
work is the construction of higher-genus exceptional covers of the
projective line. At least for covers with dihedral monodromy, Fried
mentions an interpretation in terms of moduli spaces of higher-genus
hyperelliptic curves which may lead to a similar algorithm using
isogenies of higher-dimensional abelian varieties.

Finally, an intriguing, if somewhat theoretical, question is the proper
security analysis of the trapdoor permutation described in
\S\ref{sec:trapdoor}.

\label{sec:concl}

%

\end{document}